\documentclass[11pt]{amsart}
\usepackage[leqno]{amsmath}
\usepackage{amssymb,latexsym,soul,cite,amsthm,color,enumitem,graphicx,mathtools,microtype,accents}
\usepackage[colorlinks=true,urlcolor=orange,citecolor=orange,linkcolor=orange,linktocpage,pdfpagelabels,bookmarksnumbered,bookmarksopen]{hyperref}
\definecolor{orange}{rgb}{1.0, 0.5, 0.0}
\usepackage[english]{babel}
\usepackage[left=2.7cm,right=2.7cm,top=2.4cm,bottom=2.5cm]{geometry}

\numberwithin{equation}{section}

\newtheorem{theorem}{Theorem}[section]
\theoremstyle{plain}
\newtheorem{lemma}[theorem]{Lemma}
\theoremstyle{plain}

\theoremstyle{plain}

\theoremstyle{definition}
\newtheorem{remark}[theorem]{Remark}

\newcommand{\N}{{\mathbb N}}
\newcommand{\R}{{\mathbb R}}
\newcommand{\eps}{\varepsilon}
\newcommand{\beq}{\begin{equation}}
\newcommand{\eeq}{\end{equation}}
\renewcommand{\le}{\leqslant}
\renewcommand{\ge}{\geqslant}

\newcommand{\fl}{(-\Delta)^s\,}

\def\XXint#1#2#3{{\setbox0=\hbox{$#1{#2#3}{\int}$ }
\vcenter{\hbox{$#2#3$ }}\kern-.6\wd0}}

\makeatletter
\newcommand{\leqnomode}{\tagsleft@true}
\newcommand{\reqnomode}{\tagsleft@false}
\makeatother

\newenvironment{enumroman}{\begin{enumerate}

}{\end{enumerate}}

\title[Nonlinear equations with mixed operators]{EXISTENCE OF SOLUTIONS FOR NONLINEAR EQUATIONS WITH MIXED LOCAL AND NONLOCAL OPERATORS}

\author[A.\ Iannizzotto]{Antonio Iannizzotto}

\address[A.\ Iannizzotto]{Dipartimento di Matematica e Informatica
\newline\indent
Universit\`a degli Studi di Cagliari
\newline\indent
Via Ospedale 72, 09124 Cagliari, Italy}
\email{antonio.iannizzotto@unica.it}

\subjclass[2010]{35R11, 35A15.}
\keywords{Logistic equation, Mixed operator, Variational methods.}

\begin{document}

\begin{abstract}
We study an elliptic equation, with homogeneous Dirichlet boundary conditions, driven by a mixed type operator (the sum of the Laplacian and the fractional Laplacian), involving a parametric reaction and an undetermined source term. Applying a recent abstract critical point theorem of Ricceri, we prove existence of a solution for a convenient source and small enough parameters.
\end{abstract}

\maketitle

\begin{center}
{\em Dedicated with admiration and gratitude to Prof. Biagio Ricceri -- La semplicit\`a \`e la forma della vera grandezza (F. De Sanctis).}
\end{center}

\section{Introduction and main result}\label{sec1}

\noindent
Elliptic equations driven by mixed local and nonlocal operators have been the subject of a number of interesting studies in the latest years, partly because they are used to describe the superposition of two different stochastic processes (Brownian motion and L\'evy flight), and partly due to a genuinely mathematical interest. The model for such operators is the sum of a Laplacian and a fractional Laplacian, with order $s\in(0,1)$, the latter being defined for all $u:\R^N\to\R$ as
\[\fl u(x) = C_{N,s}\,{\rm PV}\int_{\R^N}\frac{u(x)-u(y)}{|x-y|^{N+2s}}\,dy,\]
where 'PV' stands for 'principal value' and the multiplicative constant is
\[C_{N,s} = \Big[\int_{\R^N}\frac{1-\cos(z_1)}{|z|^{N+2s}}\,dz\Big]^{-1}.\]
In such case, a linear nonhomogeneous equation with homogeneous Dirichlet type conditions, set in a bounded domain $\Omega\subset\R^N$ with a general source term $h\in L^2(\Omega)$, will read as
\[\begin{cases}
-\Delta u+\fl u = h(x) & \text{in $\Omega$} \\
u = 0 & \text{in $\Omega^c$,}
\end{cases}\]
where we notice that the conditions hold on $\Omega^c=\R^N\setminus\Omega$, rather than just on $\partial\Omega$, due to the nonlocal character of the operator $\fl$. The leading operator is a linear integro-differential one. We say that $u\in H^1(\R^N)$ is a weak solution of the problem above, if $u=0$ a.e.\ in $\Omega^c$ and for all $v\in C^\infty_c(\Omega)$
\[\int_\Omega\nabla u\cdot\nabla v\,dx+C_{N,s}\iint_{\R^N\times\R^N}\frac{(u(x)-u(y))(v(x)-v(y))}{|x-y|^{N+2s}}\,dx\,dy = \int_\Omega h(x)v\,dx.\]
We refer the reader to \cite{BDVV} for an introduction to mixed operators, and to \cite{DNPV} for the fractional Laplacian and the related Sobolev spaces. Also, see \cite{BMV} for a nonlinear extension of the operator above, and \cite{DPLV} for the (non-obvious) formulation of a Neumann type boundary condition. Finally, \cite{MBRS} presents a general framework for purely nonlocal elliptic equations, seen in a variational perspective.
\vskip2pt
\noindent
In this short note, we let aside the delicate intertwining of local and nonlocal dynamics, which mostly affect the regularity of solutions, and we hold on to a simple approach, which consists in treating $\fl u$ as a lower order term in the equation. On the other hand, we consider a more general reaction involving, beside the source $h$, a generalized logistic type reaction depending on two powers $q,r\in(1,2^*)$ and a positive parameter $\lambda$:
\beq\label{log}
\begin{cases}
-\Delta u+\fl u = \lambda|u|^{q-2}u-|u|^{r-2}u+h(x) & \text{in $\Omega$} \\
u = 0 & \text{in $\Omega^c$.}
\end{cases}
\eeq
Purely nonlocal equations of the type above, with $h=0$, have been considered in \cite{IM,IMP}. In both works, bifurcations results are proved. In \cite{IM} the sublinear case $1<q<r<2$ is studied, and it is proved that positive solutions exist if, and only if, $\lambda$ lies above a threshold $\lambda_*>0$. In \cite{IMP}, an analogous result is proved for the supercritical regime $2<q<r$. Here we let $q$, $r$ span the whole subcritical interval, and we prove existence of at least one solution for $\lambda$ {\em below} a positive threshold and for a convenient smooth source $h$. In addition, we localize solutions in the unit ball with respect to the fractional seminorm:

\begin{theorem}\label{cor}
Let $\Omega\subset\R^N$ ($N\ge 2$) be a bounded domain with a $C^1$ boundary, $r\in(1,2^*)$. Then, there exists $h\in C^\infty_c(\Omega)$ with the following property: for all $q\in(1,2^*)$ there exists $\lambda^*>0$ s.t.\ for all $\lambda\in[0,\lambda^*]$ problem \eqref{log} has at least one solution $u$ with
\[\iint_{\R^N\times\R^N}\frac{|u(x)-u(y)|^2}{|x-y|^{N+2s}}\,dx\,dy < 1.\]
\end{theorem}

\noindent
In fact, we will prove a more general result, with the pure powers replaced by non-autonomous reactions with subcritical growth (see Theorem \ref{main} below). Note that the general powers considered prevent the use of the direct method of the calculus of variations, while the lack of a control at zero makes it delicate to apply min-max schemes (such as the mountain pass theorem), which anyway do not ensure localization in general. On the other hand, Theorem \ref{cor} yields no information on the sign of the solution.
\vskip2pt
\noindent
Our approach is entirely based on a recent work of Ricceri \cite{R}. In such work, two 'unusual' existence results are proved for nonlinear elliptic equations affected by another type of nonlocality, consisting in the presence of a Kirchhoff type potential. Both exploit an abstract theorem (see Theorem \ref{ric} below), which ensures existence of critical points of a perturbed functional, localized in the interior of a closed convex subset $K$ of a Banach space $X$. Here the qualifying assumption is that $K$ has nonempty interior and a nonconvex, sequentially weakly closed boundary. In \cite{R}, such $K$ is defined as a sublevel set for a sequentially weakly continuous integral functional, so $\partial K$ is the corresponding level set, which ensures closure. We have chosen as $K$ the closed unit ball in the solution space, induced by a fractional order seminorm, which turns out to be a sequentially weakly continuous functional due to an {\em ad hoc} compact embedding result (see Lemma \ref{cmp}), thus implying the localizazion in Theorem \ref{cor}.

\section{Background}\label{sec2}

\noindent
We establish a variational framework for (a generalized version of) problem \eqref{log}. Let $\Omega$ be a bounded domain in $\R^N$ ($N\ge 2$) with a $C^1$-smooth boundary $\partial\Omega$, so $\Omega$ is an extension domain for Sobolev type functions. For all $p\in[1,\infty]$ we denote by $\|\cdot\|_p$ the standard norm of $L^p(\R^N)$, and we systematically identify functions defined in $\Omega$ with their $0$-extensions to $\R^N$.
\vskip2pt
\noindent
Our solutions space consists of the $0$-extensions to $\R^N$ of functions in $H^1_0(\Omega)$. Equivalently, such space can be defined as
\[X = \big\{u\in H^1(\R^N):\,u=0 \ \text{a.e.\ in $\Omega^c$}\big\},\]
which is a separable Hilbert space with norm $\|u\|=\|\nabla u\|_2$. We recall that $X$ is compactly embedded into $L^p(\R^N)$ for all $p\in[1,2^*)$, with
\[2^* = \begin{cases}
\displaystyle\frac{2N}{N-2} & \text{if $N>2$} \\
\infty & \text{if $N=2$.}
\end{cases}\]
Also, given $s\in(0,1)$, we recall the definition of the (Gagliardo) $s$-fractional seminorm
\[[u]_s = \Big[\iint_{\R^N\times\R^N}\frac{|u(x)-u(y)|^2}{|x-y|^{N+2s}}\,dx\,dy\Big]^\frac{1}{2},\]
and we define the fractional Sobolev space $H^s(\R^N)$ as the space of all $u\in L^2(\R^N)$ with $[u]_s<\infty$ (see \cite{DNPV}). To incorporate the Dirichlet condition, we define
\[Y = \big\{u\in H^s(\R^N):\,u=0 \ \text{a.e.\ in $\Omega^c$}\big\},\]
again a Hilbert space with norm $[u]_s$ (see, for instance, \cite[Eq.\ (1.65)]{MBRS}). The variational formulation of the fractional Laplacian, in our framework, is the following: for all $u\in Y$, $\fl u$ is the gradient of the functional
\[u \mapsto \frac{C_{N,s}}{2}[u]_s^2,\]
where $C_{N,s}>0$ is defined as in Section \ref{sec1}. Equivalently, for all $u,v\in Y$
\[\langle\fl u,v\rangle = C_{N,s}\iint_{\R^N\times\R^N}\frac{(u(x)-u(y))(v(x)-v(y))}{|x-y|^{N+2s}}\,dx\,dy.\]
We could not find a direct proof in the literature, so we include a short proof of the following simple compactness result:

\begin{lemma}\label{cmp}
The space $X$ is compactly embedded into $Y$.
\end{lemma}
\begin{proof}
Continuity of the embedding $X\hookrightarrow Y$ follows from \cite[Proposition 2.2]{DNPV}. Hence, we focus on compactness. From \cite[Proposition 3.4]{DNPV} we know that $Y$ can be endowed with an equivalent norm, defined via Fourier transform. Indeed, for all $u\in Y$ we have
\beq\label{cmp1}
[u]_s^2 = \frac{2}{C_{N,s}}\int_{\R^N}|\xi|^{2s}|\mathcal{F}\,u(\xi)|^2\,d\xi,
\eeq
where $\mathcal{F}$ denotes the Fourier transform. Let $(u_n)$ be a bounded sequence in $X$, then we have for some $C>0$ independent of $n$
\[\int_{\R^N}|\xi|^2|\mathcal{F}\,u(\xi)|^2\,d\xi \le C.\]
By the compact embedding $X\hookrightarrow L^2(\R^N)$, we can find $u\in X$ and a relabeled subsequence s.t.\ $u_n\to u$ in $L^2(\R^N)$. Set $v_n=u_n-u$ for all $n\in\N$, so $v_n\to 0$ in $L^2(\R^N)$. We claim that $v_n\to 0$ in $Y$. By \eqref{cmp1}, we may equivalently prove that
\[\lim_n\,\int_{\R^N}|\xi|^{2s}|\mathcal{F}\,v_n(\xi)|^2\,d\xi = 0.\]
Fix $\eps>0$. For all $n$ big enough we have
\[\|\mathcal{F}\,v_n\|_2^2 < \eps^\frac{1}{s}.\]
Also, by Young's inequality, we have for all $\xi\in\R^N$
\[|\xi|^{2s} \le \eps|\xi|^2+C\eps^\frac{s-1}{s},\]
with $C>0$ independent of $\eps$. All these inequalities yield for all $n\in\N$ big enough (and a possibly bigger $C$)
\begin{align*}
\int_{\R^N}|\xi|^{2s}|\mathcal{F}\,v_n(\xi)|^2\,dx &\le \eps\int_{\R^N}|\xi|^2|\mathcal{F}\,v_n(\xi)|^2\,dx+C\eps^\frac{s-1}{s}\int_{\R^N}|\mathcal{F}\,v_n(\xi)|^2\,dx \\
&\le C\eps.
\end{align*}
Letting $\eps\to 0$, we find the desired convergence and prove the claim. Thus, $u_n\to u$ in $Y$.
\end{proof}

\begin{remark}\label{wsp}
We believe that a more general result than Lemma \ref{cmp} above can be proved, namely, that $W^{1,p}_0(\Omega)$ is compactly embedded into $W^{s,p}_0(\Omega)$ for all $p>1$. For instance, one could exploit an interpolation inequality, see \cite{VS}.
\end{remark}

\section{Proof of the main result}\label{sec3}

\noindent
We begin by recalling, for the reader's convenience, the abstract critical point theorem that we are going to apply:

\begin{theorem}\label{ric}
{\rm \cite[Theorem 2.1]{R}} Let $X$ be a reflexive Banach space, $Y$ be a normed space, $\psi:X\to Y$ be a bounded linear operator, $S$ be a weakly* dense, convex subset of $Y^*$, $K$ be a closed convex subset of $X$ s.t.\ ${\rm int}(K)\neq\emptyset$, $\partial K$ is sequentially weakly closed in $X$, $\psi(\partial K)$ is not convex in $Y$. Also let $\phi:\partial K\to\R$, $\Phi:K\to\R$ be a functional, G\^ateaux differentiable in ${\rm int}(K)$. Then, there exists $h\in S$ with the following property. For all $\Psi:K\to\R$, G\^ateaux differentiable in ${\rm int}(K)$, s.t.\ $\Psi-\varphi$ is constant in $\partial K$, $\Phi+\Psi$ is lower semicontinuous, strictly convex in $K$, and
\[\lim_{\|u\|\to\infty}\,\frac{\Phi(u)+\Psi(u)}{\|u\|} = +\infty \ \text{if $K$ is unbounded,}\]
and for all $\Upsilon:K\to\R$ sequentially weakly lower semicontinuous, G\^ateaux differentiable in ${\rm int}(K)$, there exists $\lambda^*>0$ s.t.\ for all $\lambda\in[0,\lambda^*]$ there exists $u\in{\rm int}(K)$ s.t.\
\[\Phi'(u)+\Psi'(u)+\lambda\Upsilon'(u)-h\circ\psi = 0 \ \text{in $X^*$.}\]
\end{theorem}

\noindent
We now prove our existence and localization result, of which Theorem \ref{cor} is a special case.

\begin{theorem}\label{main}
Let $\Omega\subset\R^N$ ($N\ge 2$) be a bounded domain with a $C^1$ boundary, $g:\Omega\times\R\to\R$ be a Carath\'eodory mapping s.t.\ for a.e.\ $x\in\Omega$
\begin{enumroman}
\item\label{main1} $|g(x,t)| \le C_1(1+|t|^{r-1})$ for all $t\in\R$ $(C_1>0$, $r\in(1,2^*))$;
\item\label{main2} $g(x,\cdot)$ is nonincreasing in $\R$;
\item\label{main3} $g(x,t)t \le 0$ for all $t\in\R$.
\end{enumroman}
Then, there exists $h\in C^\infty_c(\Omega)$ with the following property. For all Carath\'eodory mapping $f:\Omega\times\R\to\R$ satisfying for a.e.\ $x\in\Omega$ and all $t\in\R$
\[|f(x,t)| \le C_2(1+|t|^{q-1}) \ (C_2>0, \ q\in(1,2^*)),\]
there exists $\lambda^*>0$ s.t.\ for all $\lambda\in[-\lambda^*,\lambda^*]$ we can find a weak solution $u$ of
\beq\label{dir}
\begin{cases}
-\Delta u+\fl u = \lambda f(x,u)+g(x,u)+h(x) & \text{in $\Omega$} \\
u = 0 & \text{in $\Omega^c$,}
\end{cases}
\eeq
and in addition $[u]_s<1$.
\end{theorem}
\begin{proof}
Let $X$, $Y$ be defined as in Section \ref{sec2}. Set for all $u\in X$
\[\psi(u) = u,\]
so by Lemma \ref{cmp} $\psi:X\to Y$ is a linear compact operator. Also, set $S=C^\infty_c(\Omega)$, a weakly* dense linear subspace of the dual $Y^*$ (identified with $Y$ by Riesz' representation theorem). Set
\[K = \big\{u\in X:\,[u]_s\le 1\big\}.\]
Since $u\mapsto[u]_s$ is a convex continuous functional in $X$, then $K$ is a closed convex, unbounded set in $X$, with $0\in{\rm int}(K)$ and boundary
\[\partial K = \big\{u\in X:\,[u]_s=1\big\}.\]
For all $u\in\partial K$ we have $-u\in\partial K$ while $0\notin\partial K$, so $K$ is not convex in $Y$.
\vskip2pt
\noindent
We claim that $\partial K$ is sequentially weakly closed in $X$. Indeed, let $(u_n)$ be a sequence in $\partial K$, s.t.\ $u_n\rightharpoonup u$ in $X$. By \cite[Proposition 3.5]{B}, $(u_n)$ is bounded in $X$. Passing to a subsequence, if necessary, we have $u_n(x)\to u(x)$ for a.e.\ $x\in\R^N$. By Lemma \ref{cmp}, up to a further subsequence we have $u_n\to u$ in $Y$, hence
\[[u]_s = \lim_n\,[u_n]_s = 1,\]
which is equivalent to $u\in\partial K$.
\vskip2pt
\noindent
Set for all $u\in X$
\[\varphi(u) = \frac{C_{N,s}}{2}.\]
Also, for all $(x,t)\in\Omega\times\R$ set
\[G(x,t) = \int_0^t g(x,\tau)\,d\tau,\]
so by assumptions \ref{main2} \ref{main3} we see that $G:\Omega\times\R\to\R$ is a Carath\'eodory function, $G(x,\cdot)$ is concave, and for a.e.\ $x\in\Omega$ and all $t\in\R$
\beq\label{main4}
G(x,t) \le 0.
\eeq
Set for all $u\in X$
\[\Phi(u) = \frac{1}{2}\|u\|^2-\int_\Omega G(x,u)\,dx.\]
Then, $\Phi$ is strictly convex as the sum of a strictly convex functional and a convex one. By \ref{main1}, $\Phi\in C^1(X)$ with derivative given for all $u,v\in X$ by
\[\langle\Phi'(u),v\rangle = \int_\Omega\nabla u\cdot\nabla v\,dx-\int_\Omega g(x,u)v\,dx.\]
All the hypotheses of Theorem \ref{ric} are fulfilled, so let $h\in S$ be as in the statement of such result. Next we set for all $u\in X$
\[\Psi(u) = \frac{C_{N,s}}{2}[u]_s^2,\]
so $\Psi\in C^1(X)$ is strictly convex and, as recalled in Section \ref{sec2}, for all $u,v\in X$ we have
\[\langle\Psi'(u),v\rangle = \langle\fl u,v\rangle.\]
Also, for all $u\in\partial K$ we have
\[\Psi(u)-\varphi(u) = \frac{C_{N,s}}{2}\big([u]_s^2-1\big) = 0.\]
The functional $\Phi+\Psi:X\to\R$ is strictly convex and continuous. In addition, by \eqref{main4}, for all $u\in X\setminus\{0\}$ we have
\[\frac{\Phi(u)+\Psi(u)}{\|u\|} \ge \frac{\|u\|}{2},\]
and the latter tends to $\infty$ as $\|u\|\to\infty$. These observations make $\Psi$ a legitimate choice for Theorem \ref{ric}.
\vskip2pt
\noindent
There remains but one element to define. Set for all $(x,t)\in\Omega\times\R$
\[F(x,t) = \int_0^t f(x,\tau)\,d\tau,\]
and for all $u\in X$
\[\Upsilon(u) = -\int_\Omega F(x,u)\,dx.\]
By the growth condition on $f$ and the compact embedding $X\hookrightarrow L^q(\Omega)$, we see that $\Upsilon$ is sequentially weakly continuous in $X$, and $\Upsilon\in C^1(X)$ with
\[\langle\Upsilon'(u),v\rangle = -\int_\Omega f(x,u)v\,dx.\]
By Theorem \ref{ric}, there exists $\lambda^*>0$ s.t.\ for all $\lambda\in[0,\lambda^*]$ we can find $u\in{\rm int}(K)$ (i.e., $[u]_s<1$) s.t.\ in $X^*$
\beq\label{main5}
\Phi'(u)+\Psi'(u)+\lambda\Upsilon'(u)-h = 0.
\eeq
Replacing $f$ with $-f$, the conclusion is easily extended to the negative parameters $\lambda\in[-\lambda^*,0)$. Finally, we note that \eqref{main5} rephrases as follows: for all $v\in C^\infty_c(\Omega)$ we have
\[\int_\Omega\nabla u\cdot\nabla v\,dx+C_{N,s}\iint_{\R^N\times\R^N}\frac{(u(x)-u(y))(v(x)-v(y))}{|x-y|^{N+2s}}\,dx\,dy = \int_\Omega\big[\lambda f(x,u)+g(x,u)+h(x)\big]v\,dx.\]
Also, $u(x) = 0$ for a.e.\ $x\in\Omega$ by definition of $X$, so $u$ is a weak solution of \eqref{dir}.
\end{proof}

\noindent
For the sake of completeness, we give a proof of the special case seen in Section \ref{sec1}:
\vskip4pt
\noindent
{\em Proof of Theorem \ref{cor}.} Set for all $t\in\R$
\[g(t) = -|t|^{r-2}t, \ f(t) = |t|^{q-2}t,\]
so $g\in C^0(\R)$ satisfies conditions \ref{main1} \ref{main2} \ref{main3}, besides $f\in C^0(\R)$ satisfies the growth condition. Theorem \ref{main} then yields the existence and localization of a solution of \eqref{log}, for convenient $h\in C^\infty_c(\Omega)$ and $\lambda\in[0,\lambda^*]$. \qed
\vskip4pt
\noindent
{\bf Acknowledgement.} The author is a member of the group GNAMPA (Gruppo Nazionale per l'Analisi Matematica, la Probabilit\`a e le loro Applicazioni) of INdAM (Istituto Nazionale di Alta Matematica 'Francesco Severi') and is partially supported by the research projects {\em Partial Differential Equations and their role in understanding natural phenomena} (Fondazione di Sardegna 2023, CUP F23C25000080007) and {\em Regolarit\`a ed esistenza per operatori anisotropi} (GNAMPA 2025, CUP E5324001950001). We thank the anonymous referee for her/his kind words of appreciation, and E.\ Valdinoci for a precious nocturnal suggestion about Lemma \ref{cmp}.

\end{document}